\documentclass[11pt]{amsart}
\usepackage{cases}
\usepackage{amsfonts}
\usepackage{graphicx}
\usepackage{amssymb}
\usepackage{mathrsfs}
\usepackage{xcolor}
\usepackage{hyperref}
\usepackage{tikz}
\usepackage{comment}
\usepackage{enumerate}

\usepackage{float} 

\newtheorem{theorem}{Theorem}

\newtheorem{lemma}[theorem]{Lemma}

\theoremstyle{definition}
\newtheorem{remark}[theorem]{Remark}

\numberwithin{equation}{section}
\newtheorem{example}[theorem]{Example}

\newtheorem{thma}{Theorem}

\newtheorem{thmb}{Theorem}

\newtheorem{thmc}{Theorem}

\newtheorem{thmd}{Theorem}

\newtheorem{step}{Step}
\newtheorem{open problem}{Open Problem}
   
\newcommand{\C}{\mathbb{C}}

\newcommand{\D}{\mathbb{D}}
\newcommand{\R}{\mathbb{R}}
\newcommand{\T}{\mathbb{T}}
\newcommand{\re}{\mathrm{Re\,}}
\newcommand{\im}{\mathrm{Im\,}}

\DeclareMathOperator{\esssup}{ess\, sup}
\DeclareMathOperator*{\essinf}{ess\,inf}

\DeclareMathOperator{\BMO}{BMO}
\DeclareMathOperator{\dist}{dist}

\DeclareMathOperator{\sgn}{sgn}

\author{H. Gissy}
\author{S. Miihkinen}
\author{J. A. Virtanen}

\address{Department of Mathematics, University of Reading, England}

\thanks{Gissy was supported by a University of Reading doctoral grant. Miihkinen was supported by Emil Aaltonen Foundation. Virtanen was supported in part by Engineering and Physical Sciences Research Council grant EP/T008636/1.}

\begin{document}

\title{On the exponential integrability of conjugate functions}

\begin{abstract}
We relate the exponential integrability of the conjugate function $\tilde f$ to the size of the gap in the essential range of $f$. Our main result complements a related theorem of Zygmund. 

\bigskip

\noindent\textbf{MSC(2020):} 42A50\\
\textbf{Keywords:} exponential integrability, conjugate function, Hilbert transform, outer functions
\end{abstract}

\maketitle

\section{Introduction}
We denote by $L^p$ the usual Lebesgue spaces of functions on the unit circle $\T$ with norm $\|\cdot\|_p$. Given $f\in L^1$, let $u$ be the Poisson integral of $f$ and denote by $\tilde u$ the harmonic conjugate function of $u$, normalized so that $\tilde u(0)=0$. Then $\tilde u(z)$ has nontangential limit $\tilde f(\theta)$ almost everywhere on $\T$ and we call $\tilde f$ the conjugate function of $f$. Alternatively, the conjugate function $\tilde f$ can be defined  
as the principal value integral
\begin{equation}\label{e:transformation}
	\tilde f(\theta) = \lim_{\epsilon\to 0} \frac1{2\pi} \int_{|\theta-\varphi|>\epsilon}  \cot\left(\frac{\theta-\varphi}2\right) f(\varphi) \,d\varphi
\end{equation}
for almost every $\theta$. For further details and references, see Section~\ref{Poisson subsection} below.

The linear mapping $f\mapsto \tilde f$ is referred to as the conjugation operator. If $f$ is a trigonometric polynomial $\sum_{n=-N}^N a_n e^{in\theta}$, then $\tilde f$ is a trigonometric polynomial of the same degree
$$
	\tilde f(\theta) = \sum_{n=-N}^N -i \sgn(n) a_n e^{in\theta},
$$
where $-i\sgn$ is the Fourier multiplier associated with the conjugation operator.

When $1<p<\infty$, according to a famous theorem of M.~Riesz, there is a constant $C_p$ such that
$$
	\|\tilde f\|_p \le C_p \|f\|_p
$$
for all $f\in L^p$. In addition, although $f\in L^\infty$ does not imply that $\tilde f\in L^\infty$ (see, e.g.~\cite{Garnett}), the Hilbert transform still has very strong boundedness properties as can be seen in the following theorem, due to Zygmund~\cite{Zygmund}.

\begin{thma}[Zygmund]\label{Zygmund}
For $f\in L^\infty$ with $\|f\|_\infty \le \pi/2$ and $\lambda<1$, there is a constant $C_\lambda$ such that
\begin{equation}\label{e:Zygmund}
	\frac1{2\pi} \int_0^{2\pi} e^{\lambda |\tilde f|} < C_\lambda,
\end{equation}
and if $f$ is continuous on $\mathbb{T}$, then
\begin{equation}\label{e:continuous}
	\frac1{2\pi} \int_0^{2\pi} e^{\lambda |\tilde f|} < \infty
\end{equation}
for all $\lambda<\infty$.
\end{thma}

For the proof, see Corollary~III.2.6 of~\cite{Garnett}. It follows  
that
\begin{equation}\label{e:Garnett}
	f = f_1 + f_2,\ \|f_1\|_\infty < \pi/2,\ f_2\in C(\T) \implies \ \exp(\tilde f)\in L^1,
\end{equation}
where $C(\T)$ stands for the space of continuous functions on $\T$.

Let $E$ be a measurable subset of $\mathbb{T}$ and define
\begin{equation}\label{e:rho_E}
	\rho_E(z) = 2\chi_E(z) - 1 = \begin{cases}
	1, & z\in E\\
	-1, & z\notin \mathbb{T}\setminus E.
	\end{cases}
\end{equation}
The Lebesgue measure of $E$ is denoted by $|E|$. We note that the condition $\|f_1\|_\infty<\pi/2$ above is optimal as seen by considering an interval $E = [a,b] \subset (0,2\pi)$ and showing that 
\begin{equation}\label{e:interval example}
	\exp\Big(\tfrac\pi2\tilde \rho_E(t)\Big) = \exp\left( -\log \left| \sin\frac{t-b}2\right| + \log \left| \sin\frac{t-a}2\right|\right)
	=\left| \frac{\sin\frac{t-a}2}{\sin\frac{t-b}2}\right|,
\end{equation}
which is not integrable in any neighborhood of $b$. More generally, it follows from \eqref{e:interval example} and Theorem~III.2.7 of~\cite{Garnett}  
that $\exp(\tfrac\pi2\tilde \rho_E) \notin L^1$ whenever $E$ is a measurable subset of $\T$ with $0<|E|<2\pi$. 

Notice that the conditions of \eqref{e:Garnett} imply that if the function $f$ is real valued and has jumps, then the size of each jump is strictly less than $\pi$, while the size of each jump of $\frac\pi2\rho_E$ is exactly $\pi$ and $\exp (\frac\pi2\tilde\rho_E)\notin L^1$ with $E$ as above. Motivated by the study of the Fredholm properties of Toeplitz operators, Shargorodsky~\cite{S94} proved that if $g$ is real valued and $\inf \mathcal{R}(g) > \pi/2$, where $\mathcal{R}(g)$ stands for the essential range of $g$, then $\exp(\widetilde{g\rho_E})$ is not integrable. These observations lead to the question of whether
\begin{equation}\label{e:question}
	f\in L^\infty\ {\rm and}\ f\geq \pi/2\ {\rm a.e.} 
	\implies \exp(\widetilde{f\rho_E}) \notin L^1.
\end{equation}
Our main result answers this in the affirmative. Indeed, we give an elementary proof of the following result in Section~\ref{Proofs of the main results}.

\begin{theorem}\label{1.0}
Suppose that $f\in L^\infty$ with $f\ge \pi/2$ almost everywhere and $0<|E|<2\pi$. Then
there is a positive constant $C$ such that
\begin{equation}\label{e:1.2}
	|E_\lambda| = \left|\{t : \widetilde{f \rho_E }(t) > \lambda\}\right| \ge Ce^{-\lambda}
\end{equation}
for all $\lambda \ge 0$. In particular,
\begin{equation}\label{e:main2}
	\exp\left(\widetilde{f \rho_E}\right) \notin L^1.
\end{equation}
\end{theorem}

\begin{remark}
In the preceding theorem, the conditions on $E$ and that $f\ge \pi/2$ almost everywhere are optimal---see Remark~\ref{optimal E}.

Notice also that Theorem~\ref{1.0} does not imply Shargorodsky's result mentioned above.
\end{remark}

Previously in~\cite{MR1076949, MR0290089}, sufficient conditions for exponential integrability of $\tilde f$ were obtained in terms of the modulus of continuity of $f$ in $L^p$. In addition to these results and other intrinsic interest~\cite{Garnett, Katznelson, Wolff}, the integrability of the exponential of conjugate functions plays an important role in the spectral theory of Toeplitz and related operators~\cite{PVW13, S94, S07}, scalar Riemann-Hilbert problems \cite{ST08, SV06, V04} and their applications.

\section{Preliminaries}\label{Preliminaries}

As indicated in the introduction, our approach is elementary and based on classical results of complex analysis which are briefly discussed in this section.

\subsection{Poisson integrals}\label{Poisson subsection} For $f\in L^1$, denote by $P[f]$ the Poisson integral of $f$, that is,
\begin{equation}\label{e:Poisson}
	P[f](z) = \frac1{2\pi} \int_0^{2\pi} P_z(\theta) f(\theta)d\theta \qquad(z\in\D),
\end{equation}
where the Poisson kernel $P_z$ is defined by
$$
	P_z(\theta) = \re \frac{e^{i\theta} + z}{e^{i\theta}-z}.
$$
Recall that $P[f]$ is harmonic in $\D$ and if the function $f$ is continuous at $e^{i\theta}$, then
\begin{equation}\label{e:Poisson limit}
	\lim_{z\to e^{i\theta}} P[f](z) = f(\theta)
\end{equation}
(see Theorem~I.1.3 of~\cite{GM}). To deal with discontinuities at $e^{i\theta}$, define a cone $\Gamma_\alpha$ by
$$
	\Gamma_\alpha(e^{i\theta}) 
	= \{ z\in \D : |z-e^{i\theta}|<\alpha(1-|z|)\}
$$
for each $\alpha<1$, and recall that a function $\varphi: \D\to \C$ is said to have nontangential limit $\varphi^*(e^{i\theta})$ at $e^{i\theta}$ if
\begin{equation}\label{e:nontangential}
	\lim_{\Gamma_\alpha(e^{i\theta}) \ni z\to e^{i\theta}} \varphi(z)
	=\varphi^*(e^{i\theta})
\end{equation}
for every $\alpha<1$. Now, for any $f\in L^1$, if $u = P[f]$, then $u^* = f$ almost everywhere by Fatou's theorem.

Define
$$
	X_f(z) = \frac1{2\pi}\int_0^{2\pi} \frac{e^{i\theta}+z}{e^{i\theta}-z} f(\theta)d\theta\qquad (z\in\D).
$$
Then 
\begin{equation}\label{e:X_f}
	u(z) + i\tilde u(z) = X_f(z) = \frac1{2\pi} \int_0^{2\pi} \frac{e^{i\theta}+z}{e^{i\theta}-z} f(\theta)\,d\theta
\end{equation}
for $z\in\mathbb{D}$, where $u = P[f]$ and $\tilde u$ is the harmonic conjugate function of $u$ normalized so that $\tilde u(0)=0$. By Fatou's theorem and Lemma~III.1.1 of~\cite{Garnett},
\begin{equation}\label{e:Lemma III.1.1}
	X_f^*(\theta) = f(\theta) + i \tilde f(\theta)
\end{equation}
for almost every $\theta$. For the integral representation of $\tilde f$ given in~\eqref{e:transformation}, see Lemma~III.1.2 of~\cite{Garnett}. In particular, it follows that the principal value in~\eqref{e:transformation} exists almost everywhere. If $f\in L^\infty$, then
\begin{equation}\label{e:Poisson estimate}
	|\re X_f(z)| = |u(z)| < \|f\|_\infty.
\end{equation}

\subsection{Harmonic and subharmonic functions}  
In one of the key steps of the proof of the main theorem, we consider the Dirichlet problem of finding a unique bounded harmonic function on a simply connected domain $\Omega$ with prescribed boundary values.

If $g$ is a continuous real-valued function on $\partial \Omega$, the Dirichlet problem of finding the bounded harmonic function $u : \Omega\to \R$ such that $u = g$ on $\partial \Omega$ can be solved using the Poisson integral and the Riemann mapping theorem, which reduces the problem to the well-known case of the unit disk (see, e.g.,~\cite{GM}).

However, in our case, since the boundary functions are discontinuous (see~Lemma~\ref{lemmaboundaryvalues} below), the following more general result is needed.

\begin{thmb}\label{maximum principle and uniqueness}
Let $\Omega$ be a simply connected domain and let $g$ be a piecewise continuous function on $\partial\Omega$ with a finite number of discontinuities of the first kind at $\xi_1,\ldots \xi_k$. Then there is at most one bounded harmonic function $h$ on $\Omega$ such that $h=g$ on $\partial\Omega\setminus\{\xi_1,\ldots, \xi_k\}$. 

If such a bounded harmonic function $h$ exists, then 
$$
	\inf_{\zeta\in\partial\Omega\setminus\{\xi_1,\ldots,\xi_k\}} h(\zeta)
	\le h(z) \le \sup_{\zeta\in\partial\Omega\setminus\{\xi_1,\ldots,\xi_k\}} h(\zeta)
$$
for all $z\in \Omega$.
\end{thmb}

For the proof of the preceding result, see Theorems 5 and 6 of Section~42 of~\cite{MR1087298}.

\medskip

Regarding the values of a harmonic function in a domain $\Omega$ (a nonempty open connected set),  we recall the maximum principle (Theorem 1.8 of \cite{Axler}):

\begin{thmc}\label{maxprinciple}
Suppose $\Omega$ is a domain,  
$u$ is real-valued and harmonic on $\Omega$, and $u$ has a maximum or a minimum in $\Omega$. Then $u$ is constant.
\end{thmc}
Lemmas~\ref{lemmaboundaryvalues} and \ref{subharmonic function} below are utilized in key steps of the proof of the main result.

\begin{lemma}\label{lemmaboundaryvalues}
Suppose that $f \in L^\infty$ with $f \ge \frac{\pi}{2}$ almost everywhere. Let $\lambda \ge 0$ and 
$$
	G_\lambda = \{ |\re z|<\pi/2\}\cup \{ z : |\re z|<\|f\|_\infty\ {\rm and}\ \im z>\lambda+1\}
$$
be the domain in Figure~\ref{figure} and put $L = \partial G_\lambda \cap \{\im z>\lambda+1\}$. Then there exists a unique bounded harmonic function $v_\lambda$ on $G_\lambda$ with the boundary values
\begin{equation}\label{e:BV}
	v_\lambda(z) = \begin{cases}
	1, & z\in L,\\
	0, & z\in \partial G_\lambda\setminus \left(L\cup \{\pm \|f\|_\infty + i(\lambda+1)\}\right).
	\end{cases}
\end{equation}
\end{lemma}
\begin{proof}
Let $\tau$ be a conformal map of $G_\lambda$ onto the unit disk $\D$. Then each straight-line piece of the boundary $\partial G_\lambda$ is mapped to an arc on the circle $\T$ (see Theorem II.3.4' of \cite{MR0247039}). Now consider the Poisson integral of the function that equals $1$ on the arcs to which $\tau$ maps $L$ and $0$ on the complementary arcs. The composition of the Poisson integral with $\tau$ gives the desired harmonic function. Uniqueness and the bounds
\begin{equation}\label{e:v bounds}
	0 <  v_\lambda(z) < 1\quad{\rm for}\ z\in G_\lambda
\end{equation}
follow from Theorems~\ref{maximum principle and uniqueness} and \ref{maxprinciple}.
\end{proof}

The following characterization of subharmonic functions is also needed:

\begin{thmd}\label{subharmonic theorem}
Let $u$ be a function on $\D$ and suppose it satisfies the following conditions:
\begin{enumerate}[(i)]
\item $\infty \le u < \infty$, $u\not\equiv -\infty$,
\item $u$ is upper semi-continuous in $\D$,
\item for each $z_0\in\D$, there is an $r_0$ such that $D(z_0,r_0)\subset \D$ and
\begin{equation}\label{e:subharmonic theorem}
	u(z_0) \le \frac1{2\pi} \int_0^{2\pi} u(z_0 + re^{i\theta})\,d\theta
\end{equation}
for all $0\le r\le r_0$. 
\end{enumerate}
Then $u$ is subharmonic in $\D$.
\end{thmd}

For the proof of the preceding result, see Theorem~II.13 of~\cite{MR0414898}.

\begin{lemma}\label{subharmonic function}
Let $\lambda, f, G_\lambda$, and $v_\lambda$ be as in Lemma \ref{lemmaboundaryvalues} and $E \subset \T$ be a measurable set with $0<|E|<2\pi$. Define
$$
	X(z) =  \frac1{2\pi} \int_0^{2\pi} \frac{e^{i\theta}+z}{e^{i\theta}-z}
	\big(f\rho_E\big)(\theta)\,d\theta \qquad (z\in\D),
$$
where $\rho_E$ is defined in~\eqref{e:rho_E}. Let $\Omega_\lambda = X^{-1}(G_\lambda)$ and define $H_\lambda : \D\to\C$ by
$$
	H_\lambda(z) = \begin{cases}
	v_\lambda \circ X, & z\in \Omega_\lambda\\
	0, & z\in \D\setminus \Omega_\lambda.
	\end{cases}
$$
Then $H_\lambda$ is continuous and subharmonic with $0\le H_\lambda < 1$.
\end{lemma}
\begin{proof}
First, if $z\in \partial \Omega_\lambda\cap \D$, then $X(z)\notin G_\lambda$ but $X(z) \in \overline{G_\lambda\cap X(\D)}$, and hence $X(z) \in \partial G_\lambda$. Further, since $|\re X(z)|<\|f\|_\infty$ and $z\in \D$, $X(z)\in \{x+iy : |x|<\|f\|_\infty,\ y\le\lambda+1\}$ and so $v_\lambda(X(z)) = 0$ according to~\eqref{e:BV}). If $z_k\to z$ in $\Omega_\lambda$, then, using the continuity of $v_\lambda$ up to the point $X(z)$,  
$v_\lambda(X(z_k))\to 0$, which implies that $H_\lambda$ is continuous on $\D$. That $0\le H_\lambda < 1$ follows from $0 < v_\lambda < 1$ (see~\eqref{e:v bounds}) and the definition of the function $H_\lambda$.

The local mean value inequality in~\eqref{e:subharmonic theorem} holds for each point of $\Omega_\lambda$ because $H_\lambda=v_\lambda\circ X$ is harmonic in $\Omega_\lambda$, and it holds for each point of $\D\setminus\Omega_\lambda$ because $H_\lambda\ge 0$ equals zero there. Thus, by Theorem~\ref{subharmonic theorem}, the function $H_\lambda$ is subharmonic in $\D$.
\end{proof}

\section{Proof of Theorem \ref{1.0}}\label{Proofs of the main results}

Suppose that $f\in L^\infty$ with $f\ge\pi/2$ almost everywhere and $0<|E|<2\pi$. Let
$$
	X(z) =  \frac1{2\pi} \int_0^{2\pi} \frac{e^{i\theta}+z}{e^{i\theta}-z}
	\big(f\rho_E\big)(\theta)\,d\theta \qquad (z\in\D),
$$
where $\rho_E$ is defined in~\eqref{e:rho_E}. Then 
$$
	X^* = f \rho_E + i\widetilde{f\rho_E}
$$
almost everywhere according to~\eqref{e:Lemma III.1.1}.

\medskip

The proof of \eqref{e:1.2} consists of several steps.

\begin{step}\label{step1} 
Since $\re X^* = f\rho_E \ge \pi/2$ almost everywhere on $E$ and $\re X^* \le -\pi/2$ almost everywhere on $\T\setminus E$, there is a point $w\in\D$ such that $|\re X(w)|<\pi/2$. \end{step}

\begin{step} 
 
Let the open set $G_\lambda$ (see Figure~\ref{figure}) and $L = \partial G_\lambda \cap \{\im z>\lambda+1\}$ be defined as in Lemma \ref{lemmaboundaryvalues}. 
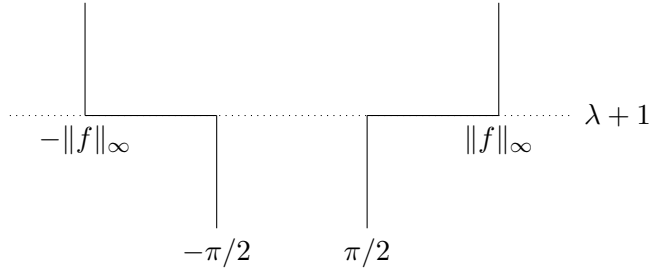
\begin{figure}[H]
\begin{tikzpicture}
\draw (2, 1.5) -- (2,0) node[anchor=north] {$-\pi/2$};
\draw (2, 1.5) -- (0.25, 1.5);
\draw (0.25, 3) -- (0.25, 1.5) node[anchor=north] {$-\|f\|_\infty$};

\draw (4,1.5) -- (4,0) node[anchor=north] {$\pi/2$};
\draw (4,1.5) -- (5.75,1.5);
\draw (5.75,3) -- (5.75,1.5) node[anchor=north] {$\|f\|_\infty$};

\draw[dotted] (-0.75,1.5) -- (6.75,1.5); 
\node[right] at (6.75, 1.5) {$\lambda+1$};
\end{tikzpicture}\caption{The open set $G_\lambda$}\label{figure}
\end{figure}

Then Lemma \ref{lemmaboundaryvalues} asserts that there exists 
a unique bounded harmonic function $v_\lambda$ on $G_\lambda$ with the boundary values
\begin{equation*}
	v_\lambda(z) = \begin{cases}
	1, & z\in L,\\
	0, & z\in \partial G_\lambda\setminus \left(L\cup \{\pm \|f\|_\infty + i(\lambda+1)\}\right).
	\end{cases}
\end{equation*}
\end{step}

\begin{step} Let $\Omega_\lambda = X^{-1}(G_\lambda)$. Then, according to Lemma \ref{subharmonic function}, the function $H_\lambda : \D\to\C$ defined by
$$
	H_\lambda(z) = \begin{cases}
	v_\lambda \circ X, & z\in \Omega_\lambda\\
	0, & z\in \D\setminus \Omega_\lambda
	\end{cases}
$$
is continuous and subharmonic with $0\le H_\lambda < 1$.

\end{step}

\begin{step}
There is a $C>0$, independent of $\lambda$, such that
\begin{equation}\label{e:vX}
	|E_\lambda| \ge C(v_\lambda \circ X)(w),
\end{equation}
where $E_\lambda$ is defined by \eqref{e:1.2} and $w$ is defined in Step 1.   
Since $H_\lambda$ is bounded, it trivially has a harmonic majorant, and hence, by Theorem I.6.7 of \cite{Garnett},
\begin{equation}\label{e:H estimate}
	H_\lambda(z) \le 
	\lim_{r\to 1} \frac1{2\pi} \int_0^{2\pi} P_z(\theta) H_\lambda(re^{i\theta})\,d\theta
\end{equation}
for $z\in\D$. It follows from~\eqref{e:Lemma III.1.1} that, for almost every $\theta\in [0,2\pi]\setminus E_\lambda$,  
$$
	|\re X^*(\theta)| = |f(\theta)|\ge \pi/2\quad{\rm and}\quad \im X^*(\theta) \le \lambda.
$$
Therefore, using the definition of $H_\lambda$ and the properties of $v_\lambda,$
$$
	\lim_{r\to 1} H_\lambda(re^{i\theta})= v_\lambda\left(X^*(\theta)\right)=0.$$ 
Now, by \eqref{e:H estimate} and Lebesgue's dominated convergence theorem,
\begin{align*}
	H_\lambda(w) &\le \frac1{2\pi} \int_{E_\lambda} P_w(\theta)\,d\theta 
	+ \frac1{2\pi} \int_{[0,2\pi]\setminus E_\lambda} P_w(\theta)  \lim_{r\to 1} H_\lambda(re^{i\theta})\,d\theta\\
	&=\frac1{2\pi} \int_0^{2\pi} P_w(\theta) \chi_{E_\lambda}(\theta)\,d\theta
	\le C \int_0^{2\pi} \chi_{E_\lambda}(\theta)\,d\theta = C|E_\lambda|,
\end{align*}
where the constant is independent of $\lambda$.
\end{step}

\begin{step} 
It is difficult to obtain the desired estimate $(v_\lambda \circ X)(w) \ge Ce^{-\lambda}$ directly. Instead we estimate $v_\lambda$  
from below by another harmonic function $g_\lambda$, defined on a vertical strip, which allows us to compute $g_\lambda$ explicitly in the next step. 

Let $g_\lambda$ be the bounded harmonic function in the strip $S=\{|\re z|<\pi/2\}$ with the boundary values
\begin{equation*}
	g_\lambda(\pm \pi/2 + iy) = \begin{cases}
	1, & y>\lambda+2\\
	0, & y<\lambda+2.
	\end{cases}
\end{equation*}
Notice that $\lim_{y\to\infty} v_0(\pm\pi/2 + i y) = 1$ (which can be verified using a conformal map of $\Omega_0$ onto $\D$). Therefore, since $v_0$ is continuous,  
Theorem~\ref{maxprinciple} implies that there is a positive constant $C$ such that 
$$
	v_0(\pm \pi/2 + iy) \ge C \ge Cg_0(\pm\pi/2+iy)
$$
for all $y>2$. Clearly $v_0(\pm\pi/2+iy) \ge 0 = Cg_0(\pm\pi/2+iy)$ for all $y<2$. Thus, by Theorem~\ref{maximum principle and uniqueness}, $v_0(z) - Cg_0(z) \ge 0$ for all $z\in S$, and so for $\lambda>0$ and $z\in S$,
$$
	v_\lambda(z) = v_0(z-i\lambda) \ge Cg_0(z-i\lambda) = Cg_\lambda(z). 
$$
Consequently, by \eqref{e:vX},
$$
	|E_\lambda| \ge C(g_\lambda\circ X)(w),
$$
where the constant $C$ is independent of $\lambda$.
\end{step}

\begin{step}
Let $\tau=x+iy = X(w)$, so $|x|<\pi/2$ and $y\in\R$. We show that there is a constant $C>0$, independent of $\lambda$, such that $g_\lambda(\tau)\ge Ce^{-\lambda}$, which completes the proof.

$$F(z) = \tan \tfrac{1}{2} z = -i\left(\frac{e^{\frac{1}{2}iz}-e^{-\frac{1}{2}iz}}{e^{\frac{1}{2}iz}+e^{-\frac{1}{2}iz}}\right)$$ defines a conformal mapping from $S$ onto $\D$ with $F(ib)= i \tanh(\frac{1}{2}
b) \to \pm i,$ as $b \to \pm\infty$. 
Notice that
$$
	F^{-1}(z) = 2\arctan z = i\log\frac{1-iz}{1+iz}
$$
for $z\in\D$. Thus, using the mapping $z\mapsto -iz$, we see that
$$
	z\mapsto i\log \frac{1-\tan\frac{z}2}{1+\tan\frac{z}2} 
	= -i\log \frac{1+\tan\frac{z}2}{1-\tan\frac{z}2}
$$
maps $S$ conformally onto itself with $\pm \frac\pi2 \mapsto (0,\mp \infty)$ and $(0,\infty)\mapsto \frac\pi2$. Therefore,
\begin{align*}
	g_\lambda(z)
	&=\frac12 + \frac1\pi \re\left(-i\log \frac{1+\tan\frac{z-i(\lambda+2)}2}{1-\tan\frac{z-i(\lambda+2)}2}\right)\\
	&=\frac12 + \frac1\pi\arg \frac{1+\tan\frac{z-i(\lambda+2)}2}{1-\tan\frac{z-i(\lambda+2)}2}.
\end{align*}
To evaluate $g_\lambda$ at $\tau=x+iy$, write $\tau_\lambda = \frac{\tau-i(\lambda+2)}2$ and notice that
\begin{align*}
	g_\lambda(\tau)
	&=\frac12 + \frac1\pi\arg \frac{1+\tan \tau_\lambda}{1-\tan \tau_\lambda}
	=\frac12+\frac1\pi\arg \tan (\tau_\lambda+\pi/4)\\
	&=\frac12 + \frac1\pi\arctan\frac{\sinh(y-(\lambda+2))}{\sin (x+\pi/2)}
	=\frac12-\frac1\pi\arctan \frac{\sinh (\lambda+2-y)}{\cos x}
\end{align*}
using the formula $\arg \tan(a+ib) = \arctan\frac{\sinh 2b}{\sin 2a}$ and the fact that $\arctan$ is odd. 
Observe also that the expression for $g_\lambda(\tau)$ is valid in $S$, but not necessarily on $\partial S$ and the first expression for $g_\lambda(z)$ is valid in $\bar{S}$. We first established the expression for $g_\lambda(z)$ with the desired boundary behavior on $\partial S$ and then considered only the behavior inside $S$, where the expressions for $g_\lambda(z)$ and $g_\lambda(\tau)$ coincide. 
Since $\arctan a = \frac\pi2-\arctan a^{-1}$ for $a>0$, we have, for $\lambda+2>y$,
$$
	g_\lambda(\tau) = \frac1{\pi} \arctan\frac{\cos x}{\sinh(\lambda+2-y)},
$$
which gives the estimate. This completes the proof of \eqref{e:1.2}.
\end{step}

It remains to prove~\eqref{e:main2}. As in Section 4 of Chapter I in \cite{Garnett}, consider the distribution function
$$
	m(\lambda) = \left| \left\{ t : \exp\left(\widetilde{f\rho_E}(t)\right)>\lambda\right\} \right| 
$$
for $\lambda>0$. Notice that for $\lambda\geq 1$,
$$
	m(\lambda) = \left| E_{\log \lambda} \right|
$$
(see~\eqref{e:1.2} for the definition of $E_{\log \lambda}$), and so by Lemma I.4.1 in \cite{Garnett},
\begin{equation}\label{e:d}
	\int \exp(\widetilde{f\rho_E}(\theta))\,d\theta = \int_0^\infty m(\lambda)\, d\lambda 
	\ge \int_1^\infty \left|E_{\log\lambda}\right| d\lambda.
\end{equation}
It remains to combine \eqref{e:1.2} with \eqref{e:d}.

\section{Further remarks}

In this section we provide remarks and examples related to Theorem~\ref{1.0}. 
We show first that the conditions in the theorem 
are optimal.

\begin{remark}\label{optimal E}
In Theorem~\ref{1.0}, 
(i) the condition on $E$ is optimal, and (ii) the condition that $f\ge \pi/2$ almost everywhere is optimal.
\end{remark}
\begin{proof} 
(i) Let $E=[0,2\pi]$ and $f=\pi/2$ on $E$. Then $f\rho_E = f$ and so trivially $e^{\widetilde{f\rho_E}}\in L^1$ by \eqref{e:continuous}.

Further, by Corollary III.1.8 of \cite{Katznelson}, 
\begin{equation}\label{e:upper bound}
	|E_\lambda| = \left| \{ t : |\widetilde{f\rho_E}(e^{it})|>\lambda\}\right| 
	= \left| \{ t : |\tilde\chi_E(e^{it})|> 2\lambda/\pi\}\right|
	< 10\pi e^{-2\lambda},
\end{equation}
which implies that there is no constant $C$ for which \eqref{e:1.2} holds in this case.

(ii) Suppose that $\pi/4<f<\pi/2$ on an interval $I=(a,b)\subset (0,2\pi)$ and $f=\pi/2$ on $[0,2\pi]\setminus I$. Let $E = [0,a+\delta] \cup [b-\delta,2\pi]$ for some small $\delta>0$. Then $\exp(\widetilde{f\rho_E})$ is integrable by Zygmund's Theorem~A because the gap in the essential range of $f\rho_E$ is strictly less than $\pi$. Consequently, the condition is optimal for \eqref{e:main2}, and it must also be optimal for \eqref{e:1.2} because it was used to prove \eqref{e:main2}. 
\end{proof}

\begin{remark}
In addition to the example in the previous proof, there are functions $f$ which are not constant and still satisfy $|E_\lambda| \lesssim e^{-\lambda}$ as in~\eqref{e:upper bound}. Indeed, let $0<|E|<2\pi$ and let $f$ be H\"older with $f(\theta) = \pi/2$ for all $\theta\in \partial E$. If $g=(f-\pi/2)\rho_E$, then it is not difficult to see that $g$ is H\"older and hence $\tilde g$ is H\"older. Also,
$$
	\widetilde{f\rho_E} - \tilde g = \tfrac\pi2\tilde\rho_E.
$$
Thus, by Exercise~VI.18 of~\cite{Garnett},
$$
	\left| \{ t : |\widetilde{f\rho_E}(t)-\tilde g(t)|>\lambda\}\right|
	=\left| \{ t : |\tfrac\pi2\tilde\rho_E(t)|>\lambda\}\right| 
	\lesssim e^{-\lambda},
$$
which implies that $|E_\lambda| = \left| \{ t : |\widetilde{f\rho_E}(t)|>\lambda\}\right| \lesssim e^{-\lambda+\|\tilde g\|_\infty} \le Ce^{-\lambda}$ for some constant $C$.

In fact, (0.2) of~\cite{Wolff}, when $\psi = \widetilde{f\rho_E} = \tilde g + \frac\pi2\tilde\rho_E$, implies the stronger result
\begin{equation}\label{e:Wolff}
	\sup_I \frac{\left|\{ t\in I : |\psi(t) - \psi_I|>\lambda\}\right|}{|I|} \lesssim e^{-\lambda}.
\end{equation}
Notice that the converse is not true, however, that is,~\eqref{e:Wolff} does not imply that 
\begin{equation}\label{e:Wolff0.2}
	\psi = u + \tilde v,\quad u\in L^\infty, \|v\|_\infty\le\pi/2
\end{equation}
(see Wolff's counterexample on page 52 of~\cite{Wolff}). 
\end{remark}

\begin{open problem}
In the preceding remark, when $\psi = \frac12\log |H|$ with $H$ univalent and zero free, Baernstein~\cite{Baernstein, Wolff} posed a question of whether~\eqref{e:Wolff} implies~\eqref{e:Wolff0.2}. This seems to be still open.
\end{open problem}

Our next example concerns outer functions from the theory of Hardy spaces. Recall that an outer function is a function $G$ on the unit disk which can be written in the form 
\begin{equation}
	G(z) = \alpha \, \exp\left(\frac{1}{2\pi} \int_0^{2\pi}\frac{e^{it}+z}{e^{it}-z}
	\log\varphi(t)\, dt\right), \quad \alpha \in \C, \, z \in \D,
\end{equation}
where $|\alpha| = 1$ and $\varphi$ is a positive measurable function on $\T$ such that $\log \varphi\in L^1$. Similarly to~\eqref{e:X_f},
\begin{equation}\label{e:outer2}
	G(z) =e^{u(z)+iv(z)}, \quad z \in \D,
\end{equation}
where $u$ is the Poisson integral of $\log \varphi$ and $v$ is the harmonic conjugate function of $u$ so that $e^{iv(0)}=\alpha$.

Let $0<p<\infty$ and let $f$ be analytic in $\D$. Then the function $f$ is in the Hardy space $H^p$ if
$$
	\sup_{r<1} \int_0^{2\pi} |f(re^{it})|^p dt < \infty.
$$
Notice that $G\in H^p$ if and only if $\varphi\in L^p$ (see, e.g., Section II.4 of~\cite{Garnett}). We can now use Theorem~\ref{1.0} to determine when certain outer functions are not in $H^p$ as shown in the following example.

\begin{example}
Given a real-valued function $f$ in $L^\infty$, define 
$$
	\Phi_f(z) = \exp\left(\frac{i}{4\pi} \int_0^{2\pi} \frac{e^{it} + z}{e^{it}-z} f(t)\, dt\right), \quad |z|<1.
$$
Using~\eqref{e:outer2}, it is easy to see that $\Phi_f$ and its inverse $\Phi_f^{-1}$ are both outer functions. 
Denote by $\mathcal{R}(f)$ the essential range of $f$ as before. Let 
$$
	I\subset [\essinf_{t\in\T} f(t), \esssup_{t\in \T} f(t)]
$$
be an interval such that $I\cap \mathcal{R}(f) = \emptyset$ and $|I| \geq \frac{2\pi}p$. Then $\Phi_f, \Phi_f^{-1}\notin H^p$. In particular, $\Phi_f^{\pm 1} \notin H^2$ if $\mathcal{R}(f)$ has a gap of length $\ge \pi$.

To see this, notice first that (similarly to \eqref{e:Lemma III.1.1})
$$
	(\Phi_f^{\pm 1})^*(t) = \exp\left(\mp \frac12 \left(\tilde f(t) - if(t)\right)\right),
$$
so $|(\Phi_f^{\pm 1})^*(t)| = e^{\mp\frac12\tilde f(t)}$. Since $\Phi_f$ is an outer function, it follows from Theorem~\ref{1.0} that $\Phi_f, \Phi_f^{-1}\notin H^p$ if $\frac{p}2|I|\ge \pi$.
\end{example}

To illustrate the effect of jumps in relation to exponential integrability (see Example~\ref{jump of pi}), the following lemma will be needed. Its proof is included for completeness because we have not found it in the literature.

\begin{lemma}\label{series}
For $0<x<2\pi$, let
\begin{equation*}
	g(x) = \sum_{n=2}^\infty \frac{\cos nx}{n\log n}.
\end{equation*}
Then 
\begin{equation}\label{e2:series}
	g(x) = \log\log \frac1x + B + O\left((\log\log\frac1x)^{1/3}(\log \frac1x)^{-2/3}\right)
\end{equation}
as $x\to 0+$, where $B$ is a positive constant.
\end{lemma}
\begin{proof}
Define $B_x(y) = \sum_{2\le n \le y} \cos nx$. Then, for $0<x<\pi$ and $y>2$, $|B_x(y)| \le \frac\pi{x} + 1$. Therefore, integrating by parts, 
$$
	\sum_{n=N}^M \frac{\cos nx}{n\log n} = \int_N^M \frac{dB_x(y)}{y\log y}
	=\frac{B_x(M)}{M\log M} - \frac{B_x(N)}{N\log N} + \int_N^M \frac{B_x(y)(\log y+1)}{y^2|\log y|^2}dy,
$$
and so
\begin{align*}
	\left|\sum_{n=N}^\infty \frac{\cos nx}{n\log n}\right|
	&\le\frac{\tfrac\pi{x}+1}{N\log N}+ \left(\tfrac\pi{x}+1\right)\int_N^\infty \frac{\log y+1}{y^2\log^2 y}dy\\
	&\le 2\frac{\tfrac\pi{x}+1}{N\log N} \le \frac{4\pi}{xN\log N}.
\end{align*}
Suppose now that $x<\tfrac\pi{2N}$. Then
\begin{align*}
	\sum_{n=2}^N \frac1{n\log n} \ge \sum_{n=2}^N \frac{\cos nx}{n\log n}
	\ge \cos Nx \sum_{n=2}^N \frac1{n\log n}
	\ge\left( 1 - \tfrac{x^2N^2}2\right)\sum_{n=2}^N \frac1{n\log n}.
\end{align*}
Since 
$$
	\sum_{k=2}^n \frac1{k\log k} = \log\log n + B + O(\frac1{n\log n}),
$$
where $B$ is a constant (see Exercise~8.20 of~\cite{Apostol}), we get 
for $x<\frac{\pi}{2N}$,
\begin{equation*}
	g(x) = \log\log N + B + O\left(\frac1{N\log N}\right)
	+O\left(x^2N^2 \log\log N\right) + O\left(\frac1{xN\log N}\right)
\end{equation*}
as $N \to \infty.$  

Choosing $N = \frac1x(\log\frac1x \log\log \frac1x)^{-1/3}$, we obtain~\eqref{e2:series} as $x\to 0+$.
\end{proof}

\begin{remark}
In the following example, for small values of $|x|$, we only need the following consequence of the preceding lemma:
$$
	\sum_{n=2}^\infty \frac{\cos nx}{n\log n} \ge C \log\log (1/|x|),
$$
which can also be obtained using Theorem~V.1.5 of~\cite{Zygmund}.
\end{remark}

\begin{example}\label{jump of pi}
Let $t_0\in (0,2\pi)$. For $0<\delta<1$ define
$$
	g(t) = \begin{cases}
	\pi/2 & t_0-\delta < t < t_0\\
	-\pi/2 & t_0 < t < t_0+\delta
	\end{cases}
$$
and suppose that $g$ is H\"older continuous elsewhere. By Theorem~V.1.3 of~\cite{Z}, the series
$$
	h(t) = 2\sum_{n=2}^\infty \frac{\sin n(t-t_0)}{n\log n}
$$
converges uniformly and defines a continuous function on $[0,2\pi]$. It is well known that
$$
	\tilde h(t) = - 2\sum_{n=2}^\infty \frac{\cos n(t-t_0)}{n\log n},
$$
which is continuous on $[0,2\pi]\setminus\{t_0\}$ (see Theorem~I.2.6 of~\cite{Z}). Define $f=g+h$. Then $\lim_{t\to t_0\pm} f(t) = \mp \pi/2$, so $f$ is piecewise continuous with only one jump, which is of size $\pi$. We want to determine whether $\exp(\tilde f)$ is integrable. Obviously we cannot use Zygmund's Theorem~A. Notice also that we cannot apply Theorem~\ref{1.0} because we do not know without further inspection whether $f\ge \pi/2$ a.e.~on $(0,t_0)$ and $f\le -\pi/2$ a.e.~on $(t_0,2\pi)$.

Now, similarly to~\eqref{e:interval example}, and using the fact that $g$ is H\"older on $[0,2\pi]\setminus \{t_0\}$, there is a constant $C>0$ such that
$$
	\tilde g(t) \le -\log\left| \sin\frac{t-t_0}2\right| + C
$$
for $0\le t \le 2\pi$. To estimate $\tilde h$ near $t_0$, notice first that Lemma~\ref{series} implies that
$$
	\tilde h(t) \le -2\log(\log |t-t_0|^{-1})
$$
for $t$ sufficiently close to $t_0$. Therefore, for some constants $C$, we have
\begin{align*}
	e^{\tilde f(t)} &= e^{\tilde g(t)} e^{\tilde h(t)} 
	\le C \exp\left(-\log |\sin \frac{t-t_0}2|\right) \exp\left(-2\log\log\frac1{|t-t_0|}\right)\\
	&= C\left|\sin \frac{t-t_0}2\right|^{-1}\left(\log \frac1{|t-t_0|}\right)^{-2}
	\le C|t-t_0|^{-1} \left(\log|t-t_0|\right)^{-2},
\end{align*}
which is integrable in a neighborhood of $t_0$, and hence $e^{\tilde f}\in L^1$.

Now, of course, by Theorem~\ref{1.0}, the integrability of $e^{\tilde f}$ means that $f$ has values in $(-\pi/2, \pi/2)$ on a set of positive measure. In fact, by~V.2.13 of~\cite{Zygmund}, the function $h$ is positive on $(t_0, t_0+\epsilon)$ for some $\epsilon>0$ and hence negative on $(t_0-\epsilon,t_0)$ (as an odd function), and so indeed $f<\pi/2$ on $(t_0-\epsilon, t_0)$ and $f>-\pi/2$ on $(t_0, t_0+\epsilon)$.
\end{example}

\begin{remark}
We can use the previous example to construct a function $f\in L^\infty$ such that $\|f\|_\infty = \pi/2$, $|f|<\pi/2$ and $e^{\tilde f}\in L^1$. This should be compared with Zygmund's result in \eqref{e:Zygmund}.
\end{remark}

\begin{open problem}
In Theorem~\ref{1.0}, it is assumed that $f\in L^\infty$ as in Zygmund's Theorem~\ref{Zygmund}. It is natural to ask whether the conclusion of Theorem~\ref{1.0} remains true for unbounded functions.
\end{open problem}

We finish this section with a connection between the estimate in~\eqref{e:1.2} and the distance in $\BMO$ to $L^\infty$ (see Section~VI.6 of~\cite{Garnett}). 

Suppose that the conditions of Theorem~\ref{1.0} are satisfied, that is, $f\in L^\infty$ with $f\ge \pi/2$ almost everywhere and $0<|E|<2\pi$. If $\|g\|_\infty<\pi/2$, then by Theorems~\ref{Zygmund} and~\ref{1.0}, $\widetilde{f\rho_E} - \tilde g\notin L^\infty$, and so $\dist(\widetilde{f\rho_E}, L^\infty) \ge \pi/2$, where
$$
	\dist(\varphi, L^\infty) = \inf\{ \|g\|_\infty : \varphi-\tilde g\in L^\infty\},
$$
which is equivalent to $\inf_{g\in L^\infty} \|\varphi-g\|_{\BMO}$ (see, e.g., page 250 in~\cite{Garnett}). By Corollary VI.6.6, there is no $\epsilon\in (0,1)$ such that
\begin{equation}\label{Garnett (6.6)}
	\sup_I \frac{|\{ t\in I : |\widetilde{f\rho_E}(t)-(\widetilde{f\rho_E})_I|>\lambda\}|}{|I|}
	\le e^{-\lambda/\epsilon}
\end{equation}
for all $\lambda\ge 0$, where the supremum is taken over all arcs $I\subset\T$ and the average $\varphi_I$ is defined by $\varphi_I = \frac1{|I|}\int_I \varphi$ for $\varphi\in L^1$.

The same conclusion also follows directly from~\eqref{e:1.2} if $f\ge \pi/2$, which is no surprise because the requirement that $f$ has a gap in its essential range is a stronger assumption than $\dist(\widetilde{f\rho_E}, L^\infty) \ge \pi/2$. Indeed, assume that \eqref{e:1.2} holds and \eqref{Garnett (6.6)} does not hold, so that there exists $\epsilon \in (0,1)$ satisfying the estimate \eqref{Garnett (6.6)}. Denote $h = f\rho_E$ so that $\tilde h_0 = \frac{1}{2\pi}\int_0^{2\pi} \tilde{h}$, and choose $s \ge 0$ so large that $\lambda = \tilde h_0 +s \ge 0.$ Then there exists a constant $C > 0$ so that  
\begin{align*}
C e^{-\lambda} &\le \frac{1}{2\pi}|\{\theta \in [0, 2\pi]: \tilde{h}(\theta) > \lambda\}| = \frac{1}{2\pi}|\{\theta \in [0, 2\pi]: \tilde{h}(\theta)-\tilde{h_0} > s\}| 
\\
&\le \sup_{I \subset \T} \frac{1}{|I|}|\{\theta \in I: |(\tilde{h}-\tilde{h}_I)(\theta)| > s \}| \le e^{-\frac{s}{\epsilon}},
\end{align*}
so we have $C e^{-\lambda} = C e^{-(\tilde{h}_0+s)} \le e^{-\frac{s}{\epsilon}}.$ Therefore $Ce^{-\tilde{h}_0} \le e^{(1-\frac{1}{\epsilon})s} \to 0$ as $s \to \infty,$ which is a contradiction.

\section{Complex-valued functions}

While real-valued functions are of particular importance in the study of exponential integrability of their conjugate functions, especially in connection with applications, such as Riemann-Hilbert problems and spectral theory of Toeplitz operators above, it would also be of interest to consider the case of complex-valued functions. Indeed, as in Zygmund's Theorem A, we may consider a complex-valued $f\in L^\infty$ and ask under what conditions is $\exp(|\widetilde {f\rho_E}|)$ not integrable. As in the proof of Theorem~\ref{1.0}, we can define $m(\lambda) = |\{ t : \exp(|\widetilde{f\rho_E}(t)|)>\lambda\}|$ and show that if there is a constant $|\{t : |\widetilde{f\rho_E}(t)|>\lambda\}| \ge Ce^{-\lambda}$ for all $\lambda\ge 0$, then $\exp(|\widetilde {f\rho_E}|)$ is not integrable. 

However, in the complex case, the function $f\rho_E$ no longer has a similar (geometric) meaning as in the real case where it can be related to a gap in the essential range. 
For this reason, we say that a set $A\subset\C$ has a gap of size $g>0$ if $A = B \cup C$ for some sets $B$ and $C$ of positive measure with $\dist(A,B)\ge g$. With this, we can state \eqref{e:main2} in Theorem~\ref{1.0} as follows: If $f\in L^\infty$ is real and $\mathcal{R}(f)$ has a gap of size at least $\pi$, then $\exp(\tilde f)$ is not integrable. 

\begin{open problem}
Given a complex-valued function $f$ in $L^\infty$, find a converse to Zygmund's Theorem A.
\end{open problem}

It may be useful to try to relate the exponential integrability of $\tilde f$ to the size of the gap in the essential range of $f$ as in the real case.

\end{document}